\documentclass[11pt]{amsart}
\usepackage[T1]{fontenc}
\usepackage[english]{babel}
\usepackage{kpfonts}

\usepackage{amsmath,amssymb,amsthm,enumerate,xspace}
\usepackage[colorlinks=true,citecolor=blue,linkcolor=blue]{hyperref}
\usepackage{cleveref}% Has to be loaded after hyperref
\usepackage[all]{xy}
\usepackage{mathrsfs}

\numberwithin{equation}{section}

\newtheorem{thm}{Theorem}[section]

\newtheorem*{conj*}{Conjecture}

\newtheorem{prop}[thm]{Proposition}
\newtheorem{lem}[thm]{Lemma}
\newtheorem{cor}[thm]{Corollary}
\newtheorem*{fact*}{Fact}

\newtheorem{exam}[thm]{Example}
\newtheorem{ithm}{Theorem}

\newtheorem{iprop}[ithm]{Proposition}

\newtheorem*{iprob*}{Problem}
\newtheorem{rem}[thm]{Remark}
\newtheorem*{rem*}{Remark}

\newtheorem{irem}[ithm]{Remark}
\newtheorem{ithmbis}{Theorem}

\theoremstyle{definition}

\newtheorem{defi}[thm]{Definition}
\newtheorem{merci}[ithm]{Acknowledgements}
\newcommand{\sA}{\mathscr{A}}
\newcommand{\sB}{\mathscr{B}}

\newcommand{\sF}{\mathscr{F}}

\newcommand{\sO}{\mathscr{O}}

\newcommand{\ZZ}{\mathbf{Z}}

\newcommand{\RR}{\mathbf{R}}
\newcommand{\NN}{\mathbf{N}}
\newcommand{\GG}{\mathbf{G}}

\newcommand{\FF}{\mathbf{F}}
\newcommand{\QQ}{\mathbf{Q}}

\newcommand{\SO}{\mathbf{SO}}

\newcommand{\SL}{\mathbf{SL}}

\newcommand{\SSS}{\mathbf{S}}

\newcommand{\hb}{\mathrm{H}_{\mathrm{b}}}
\newcommand{\hc}{\mathrm{H}_{\mathrm{c}}}
\newcommand{\hcb}{\mathrm{H}_{\mathrm{cb}}}
\def\hbc{\hcb}

\newcommand{\hh}{\mathrm{H}}

\newcommand{\bu}{\bullet}

\newcommand{\simplex}{\Delta}
\newcommand{\se}{\subseteq}
\newcommand{\inv}{^{-1}}

\newcommand{\fhi}{\varphi}

\newcommand{\lra}{\longrightarrow}

\DeclareMathOperator{\Ker}{Ker}

\newcommand{\nerve}{N}
\newcommand{\order}{B}
\newcommand{\qtop}{\overline{q}}
\DeclareMathOperator{\sub}{Sub}
\DeclareMathOperator{\vertex}{Vert}

\title[Flatmates and bounded cohomology of algebraic groups]{Flatmates and the bounded cohomology \\ of algebraic groups}

\author[Nicolas Monod]{Nicolas Monod}
\address{%\'Ecole Polytechnique F\'ed\'erale de Lausanne (EPFL)\\
EPFL,
Switzerland}
\begin{document}

\begin{abstract}
For all algebraic groups over non-Archimedean local fields, the bounded cohomology vanishes. This follows from the corresponding statement for automorphism groups of Bruhat--Tits buildings, which hinges on the solution to the flatmate conjecture raised in earlier work with Bucher. Vanishing and invariance theorems for arithmetic groups are derived.
\end{abstract}

\maketitle
%\thispagestyle{empty}

%\setcounter{tocdepth}{1}
%\tableofcontents
%\newpage

\section{Introduction}%%%%%%%%%%%%%%%%%%%%%%%%%%%%%%%%%%%%%%%%%%%%%%%%%%%%%%%%%%%%%%%%%%
%%%%%%%%%%%%%%%%%%%%%%%%%%%%%%%%%%%%%%%%%%%%%%%%%%%%%%%%%%%%%%%%%%%%%%%%%%%%%%%%%%%%
Our main goal is the following vanishing theorem.

\begin{ithm}[Algebraic groups]\label{ithm:algebraic}~\\
Let $\GG$ be any algebraic group over a non-Archimedean local field~$k$.

Then the continuous bounded cohomology of $\GG(k)$ with real coefficients vanishes in every positive degree.
\end{ithm}

\noindent
It is understood that the group $G=\GG(k)$ of $k$-points of the scheme $\GG$ is endowed with the locally compact topology determined by the local field $k$. Examples are (almost-)simple linear algebraic groups such as
\[
G=\SL_d(k) \text{ or } G=\SO_{d}(k), \kern5mm \text{both with } k=\QQ_p \text{ or }  k=\FF_q((t)),
\]
for which \Cref{ithm:algebraic} can be viewed as a strengthening of the classical vanishing theorem of Garland, Casselman, Wigner and Harder~\cite{Garland, Casselman, Casselman-Wigner, Harder77}, noting that non-trivial coefficients were already treated in~\cite{MonodVT}. Structure theory and general principles will reduce \Cref{ithm:algebraic} to the case of such simple algebraic groups; in view of Bruhat--Tits theory, that case is, in turn, contained in the following statement.

\begin{ithm}[Buildings]\label{ithm:building}~\\
Let $G$ be a locally compact group acting properly by automorphisms on a locally finite Euclidean building.

If this action is strongly transitive, then the continuous bounded cohomology of $G$ with real coefficients vanishes in every positive degree.
\end{ithm}

\noindent
Previously, this statement was only known in the special case of trees~\cite{Bucher-Monod_tree}.

\subsection*{Motivations}%%%%%%%%%%%%%%%%%%%%%%%%%%%%%%%%%%%%%%%%%%%%%%%%%%%%%%%%%%%%%%%%%%
Our first motive to establish \Cref{ithm:algebraic} is that the bounded cohomology of \emph{real or complex} algebraic groups, and more generally of Lie groups, remains very mysterious to this day even though it has been subjected to intense scrutiny ever since Gromov's seminal work~\cite{Gromov}. The original impetus for that study goes back to Milnor's 1958 paper~\cite{Milnor58} and the Milnor--Wood inequality: the fact that some characteristic classes happen to be bounded translates into non-trivial estimates for topological invariants of bundles and manifolds. Gromov conceptualised this by introducing bounded cohomology and proved that all characteristic classes of flat $G$-bundles, where $G$ is a $\RR$-algebraic group, are bounded. (Another proof was given by Bucher in~\cite{BucherKarlsson,Bucher07}; see also~\cite{Hartnick-Ott12} for related results.)

Nonetheless, Dupont's 1979 conjecture~\cite{Dupont} that all the real continuous cohomology classes of connected simple Lie groups are bounded remains open for many groups. In fact, if we exclude groups of Hermitian type, it is open for almost all other simple Lie groups; see~\cite{Hartnick-Ott12}. A stronger version of this conjecture, open for \emph{all} semi-simple (non-compact) groups, is as follows:

\begin{conj*}[Problem~A in~\cite{MonodICM}]~\\
Let $G$ be a connected semi-simple Lie group with finite center. Then the continuous bounded cohomology of $G$ with real coefficients is naturally isomorphic to its ordinary (continuous) cohomology.
\end{conj*}

For semi-simple algebraic groups over non-Archimedean local fields, it is a classical result that the ordinary continuous real cohomology vanishes~\cite[Cor.~2]{Casselman-Wigner}, \cite[\S X]{Borel-Wallach}. In that sense, \Cref{ithm:algebraic} answers the analogue of the above conjecture in the non-Archimedean case.

This was previously only known for rank one groups, such as $\SL_2(\QQ_p)$, because the statement of \Cref{ithm:building} was only available in the particular case of trees~\cite{Bucher-Monod_tree}. (For higher rank groups, vanishing in degree two was obtained~\cite{Burger-Monod1} and the stability methods of~\cite{MonodVT} reduce degree three to the rank one case as proved in~\cite[Thm.~6.3.1]{GLMR_arxv4}. Nothing was known in higher degrees.)

In the Lie case originally considered for the above conjecture, the isomorphism is only known in the following low degrees. In degree $2$ by~\cite{Burger-Monod1}. In degree~$3$, for some groups by \cite{MonodJAMS,Pieters18,Bucher-Burger-Iozzi18} and very recently \cite{Mengual_deg3_arxv2}, using~\cite{Mengual-Hartnick_stab_arx,Mengual-Hartnick_quillen}, established degree~$3$ for all classical complex groups. In degree~$4$ it is known for $\SL_2(\RR)$ only~\cite{Hartnick-Ott15}.

\medskip

A second motivation for \Cref{ithm:algebraic} is that when ordinary cohomology is already known to vanish, as for simple non-Archimedean groups, then the bounded vanishing is a strict strengthening of vanishing. Indeed, it implies that even ``almost-cocycles'' must be trivial, as illustrated by the case of quasi-morphisms. That case, which concerns $n=2$ only, has a number of applications to rigidity. Higher vanishing and bounded acyclicity have been much studied recently, though mostly for ``large'' transformation groups without topology~\cite{Matsumoto-Morita,Loeh_note17,Monod_wreath,Monod-Nariman_inv,FFLM_binate,Campagnolo-Fournier-Facio-Lodha-Moraschini_arx_acc, FFF-monod-nariman_arx1}.

\subsection*{Arithmetic groups and non-trivial coefficients}%%%%%%%%%%%%%%%%%%%%%%%%%%%%%%%%%%%%%%%%%%%%%%%%%%%%%%%%%%%%%%%%%%
Our third incentive is the cohomology of \emph{discrete} groups. \Cref{ithm:algebraic} provides one of the main missing pieces for the study of the bounded cohomology of $S$-arithmetic groups. In the case of ordinary cohomology, this is the most classical motivation and the main reason why continuous cohomology of algebraic groups has been studied for general coefficients~\cite{Borel-Wallach}. Indeed, following Borel--Serre~\cite{Borel-Serre76}, to study the ``abstract'' cohomology of an $S$-arithmetic group $\Gamma$, one realises it as a lattice $\Gamma<G$ in a product of algebraic groups over various local fields (which is possible by results of Borel~\cite[\S8]{Borel63}, respectively, Behr--Harder~\cite{Behr69,Harder69} in positive characteristic). For instance,
\[
\Gamma = \SL_d(\ZZ[1/p])\kern5mm\text{in }  G=\SL_d(\RR) \times \SL_d(\QQ_p).
\]
In a suitable range of degrees, the cohomology of $\Gamma$ will then be determined by the continuous cohomology of $G$ with coefficients in an induction module. Therefore, vanishing results for $G$ with non-trivial coefficients will imply \emph{invariance theorems}, namely the statement that $\hh^n(\Gamma)$ is isomorphic to $\hc^n(G)$ for suitable $n$~\cite{Borel-Serre76,Serre71}. At that point, the vanishing for non-Archimedean groups will further indicate that $\hh^n(\Gamma)$ is given by the continuous cohomology of the Archimedean factors, which is known from Lie theory. In particular, if $G$ has no non-compact Archimedean factors, e.g., in positive characteristic, then one concludes a vanishing result for $\Gamma$ (still for suitable $n$ only), as originally conjectured by Serre.

\medskip
This classical picture has an analogue in bounded cohomology with notable differences. For non-trivial coefficients, the ordinary vanishing holds below the rank by work of Garland~\cite{GarlandICM,Garland}, Casselman~\cite{Casselman}, Casselman--Wigner~\cite{Casselman-Wigner}, Borel--Wallach~\cite{Borel-Wallach}. We established it for bounded cohomology below twice the rank~\cite{MonodVT}, by different methods but still using a form of Solomon--Tits theorem.

However, for trivial coefficients $\RR$, the vanishing of \Cref{ithm:algebraic} was previously unknown because the Bruhat--Tits building methods from ordinary cohomology fail in the bounded setting. To illustrate this, consider that ordinary vanishing \emph{above} the rank for arbitrary coefficients clearly holds since the rank is the dimension of this building, which is contractible. This sort of soft and easy principles fail completely for bounded cohomology, which is one of the reasons for its appeal, and for its difficulty. For instance, the tree of $\SL_2(\QQ_p)$ is a contractible one-dimensional simplicial complex~\cite{Serre77}, but this does not preclude degree-two bounded cohomology. Here is a folklore example:

\begin{iprop}[Coefficients]\label{iprop:SL2}~\\
There exist irreducible continuous unitary representations $\pi$ of $G=\SL_2(\QQ_p)$ with $\hbc^2(G, \pi)\neq 0$.
\end{iprop}

Returning to arithmetic groups, we obtain a vanishing theorem for these discrete groups in the spirit of Garland's results on Serre's conjecture by combining \Cref{ithm:algebraic} with our results from~\cite{MonodVT}.

\begin{ithm}[Discrete groups]\label{ithm:Sadic:0}~\\
Let $K$ be a global field and $\GG$ a connected simple linear $K$-group which is anisotropic over the Archimedean completions of $K$.

Let $S$ be a finite set of valuation classes of $K$ and let $\Gamma=\GG(K(S))$ be the corresponding $S$-arithmetic group over the ring $K(S)$ of $S$-integers.

Then $\hb^n(\Gamma)=0$ for all $0<n<2 \sum_{v\in S} \mathrm{rank}_{K_v}(\GG)$.
\end{ithm}

The assumption on Archimedean completions is trivially satisfied when $K$ has positive characteristic. In characteristic zero, a concrete example is as follows. Let $p$ be a prime~$\equiv 1$ mod~$4$ and let $d\geq 5$. Then
\[
\hb^n\big(\SO_d(\ZZ[1/p])\big) \ =0 \kern5mm \forall 0<n<  d-1
\]
and also $n=d-1$ for $d$ even, since the $\QQ_p$-rank of $\SO_d$ is $\lfloor d/2 \rfloor$, using $p\equiv 1$ mod~$4$ via Gau\ss's Theorem~108 \cite[\S IV.108]{GaussDA}. (The restriction $d\geq 5$ is to avoid the non-simple case $d=4$ and the trivial range of $n$ for $d\leq 3$.)

\smallskip
More generally, when isotropic Archimedean places are allowed, we obtain an invariance theorem:

\addtocounter{ithmbis}{\value{ithm}}
\addtocounter{ithmbis}{-1}
\begin{ithmbis}[Discrete groups, bis]\label{ithm:Sadic}~\\
Let $K$ be a global field, $\GG$ a connected simple linear $K$-group and $S$ a finite set of valuation classes of $K$ containing the set $S_0$ of all Archimedean ones for which $\GG$ is isotropic.

Let $\Gamma=\GG(K(S))$ be the corresponding $S$-arithmetic group and consider the semi-simple Lie group $L= \prod_{v\in S_0} \GG(K_v)$.

Then $\hb^n(\Gamma) \cong \hbc^n(L)$ for all $n<2 \sum_{v\in S} \mathrm{rank}_{K_v}(\GG)$.
\end{ithmbis}

In general, only the case $n=2$ was previously known: this was the main result of~\cite{Burger-Monod1}.

As a concrete example, given an integer $m>1$, the inclusion of the $S$-arithmetic group $\SL_d(\ZZ[1/m])$ into $\SL_d(\RR)$ induces an isomorphism
\[
\hb^n \big(\SL_d(\ZZ[1/m])\big) \ \cong\ \hbc^n\big(\SL_d(\RR)\big) \kern5mm \forall n< 2 (d-1) (\omega(m) + 1)
\]
where $\omega(m)$ denotes the number of distinct prime factors of $m$. Again, this was previously only known in the special case of $d=2$ by the result for trees~\cite[Cor.~4]{Bucher-Monod_tree}.

We note here that the ordinary (virtual) cohomological dimension of $\SL_d(\ZZ[1/m])$ is $n=(d-1) (\omega(m) + d/2)$ by Borel--Serre~\cite[\S6]{Borel-Serre76}, which lies in our range for $n$ in \Cref{ithm:Sadic} as soon as $m$ has more than $d/2-2$ distinct prime factors.

\begin{irem}[Equivalent formulation]~\\
Theorems~\ref{ithm:Sadic:0} and~\ref{ithm:Sadic} could instead be formulated for general irreducible lattices in semi-simple groups as all results used in the proof hold in that setting. This would however not really add any generality. Indeed, the range for $n$ is void in rank one (since $\hb^1$ always vanishes) and in rank~$\geq2$ Margulis's arithmeticity theorem shows that all lattices are commensurable to $S$-arithmetic groups. We find the above statements more concrete as they focus on the structure of the groups $\Gamma$ under consideration.
\end{irem}

\subsection*{Flatmates and our approach}%%%%%%%%%%%%%%%%%%%%%%%%%%%%%%%%%%%%%%%%%%%%%%%%%%%%%%%%%%%%%%%%%%
With Bucher, we proposed in~\cite{Bucher-Monod_tree} a strategy towards \Cref{ithm:building} and implemented it in the special case of trees. The main difficulty in that strategy is to understand the \emph{flatmate complex}. This object, which will be detailed in \Cref{sec:flatmate} below, consists of all tuples of vertices that lie in a common flat. More precisely, the question raised by Conjecture~10 in~\cite{Bucher-Monod_tree} is to determine the ``uniform homotopy type'' of this complex. We can now answer this problem as follows.

\begin{ithm}[Flatmates]\label{ithm:flatmate}~\\
The flatmate complex of any discrete irreducible Euclidean building is uniformly acyclic.
\end{ithm}

In the special case of trees, where the flatmate complex is the ``aligned complex'', this statement was established in~\cite{Bucher-Monod_tree} by exhibiting a relatively simple, geometrically meaningful, bounded homotopy. By contrast, in the present case of buildings, the combinatorics of arbitrary configurations of finitely many points seems far too complicated (for the present author). Therefore, we shall prove \Cref{ithm:flatmate} by introducing a few general simplicial tools which will lead to a solution by general principles, commuting back and forth between uniform and non-uniform homotopy arguments.

\smallskip
The new contributions of our approach are as follows. Contrary to ordinary contractibility, Euclidean buildings are not uniformly acyclic. We show in essence that they become so ``modulo their flats'' by considering the nerve of the apartment system. In order to show that this nerve is uniformly acyclic, we use on the one hand that the nerve is \emph{non-uniformly} homotopic to the building, which is \emph{non-uniformly} acyclic. On the other hand, we introduce a \emph{support control} principle which allows us to upgrade the acyclicity of the nerve to its uniform counterpart using a \emph{uniform nerve} principle that we provide. This requires a quantitative control on finite subcomplexes; using building-theoretical arguments, we establish that this control holds in the case of apartment systems (though not for Euclidean buildings themselves).

We expect these tools to be useful beyond the application to algebraic groups.

\smallskip

Regarding the uniform nerve principle, we recall that \emph{topological} nerve theorems fail catastrophically in bounded cohomology. On the one hand, the circle is boundedly acyclic but the nerve of a finite good cover of the circle is not. In the reverse direction, the nerve of a finite good cover of a bouquet of two circles has no cohomology in dimension~$>1$, while this bouquet has infinite-dimensional bounded cohomology in dimensions~$2$ and~$3$.

\begin{merci}[Gratitude]
I am very grateful to Pierre-Emmanuel Caprace and Francesco Fournier-Facio for their comments, and to the anonymous referee whose suggestions and remarks were very helpful, notably in simplifying the statements in \Cref{sec:nerve}.
\end{merci}

\section{Simplicial methods}%%%%%%%%%%%%%%%%%%%%%%%%%%%%%%%%%%%%%%%%%%%%%%%%%%%%%%%%%%%%%%%%%%
%%%%%%%%%%%%%%%%%%%%%%%%%%%%%%%%%%%%%%%%%%%%%%%%%%%%%%%%%%%%%%%%%%%%%%%%%%%%%%%%%%%%
\subsection{Notation}%%%%%%%%%%%%%%%%%%%%%%%%%%%%%%%%%%%%%%%%%%%%%%%%%%%%%%%%%%%%%%%%%%
We use the standard notation where a \textbf{simplicial complex} $\Sigma$ is a set of non-empty finite sets closed under passing to non-empty subsets. The set of $q$-simplices is denoted $\Sigma_{q}$; strictly speaking we distinguish $\Sigma_{0}$ from the vertex set $\vertex(\Sigma)=\bigcup_{\sigma\in\Sigma} \sigma$. The \textbf{hypersimplex} $\simplex X$ on a set $X$ is the simplicial complex of all non-empty finite subsets $\sigma\se X$. Note that if $X_1,X_2\se X$ then the subcomplex $\simplex X_1 \cap \simplex X_2$ coincides with $\simplex (X_1 \cap X_2)$. A subcomplex $\Sigma'\se \Sigma$ is \textbf{full} if it contains every $\sigma\in \Sigma$ with $\sigma\se \vertex(\Sigma')$. A \textbf{simplicial map} is a map $f\colon \Sigma \to \Sigma'$ that is induced by a vertex map $\vertex\Sigma \to \vertex\Sigma'$ also abusively denoted by $f$.

\medskip
Given a cover $\sF$ of a set $X$, the \textbf{nerve} $\nerve \sF \se \simplex \sF$ is the simplicial complex of all non-empty finite subsets of $\sF$ having non-empty intersection.

\medskip
Given a \textbf{poset} $S$, that is, a set endowed with a partial order, the corresponding \textbf{order complex} is the simplicial complex $\order S \se \simplex S$ consisting of all non-empty finite chains. Thus an element of $\order S_{q}$ is of the form $\{s_0 < \cdots < s_q\}$. Both isotone and antitone (i.e., order preserving/reversing) poset maps induce simplicial maps since the definition of $\order S$ is self-dual.

Our notation reflects the fact that $\order S$ represents a classifying space of $S$ viewed as a category (compare~\cite{Quillen78}). We warn the reader that $\order S$ is sometimes called a ``nerve'' (and its realisation a classifying space)~\cite{Segal68}; adding to the confusion, the nerve that we defined above is itself a classifying space of a category of inclusions associated to the cover.

Finally, since a simplicial complex $\Sigma$ is a poset under inclusion, we can form its order complex $\order\Sigma$, which is the \textbf{barycentric subdivision} of $\Sigma$.

\medskip
We write $C_q(\Sigma)$ for the group of real-valued $q$-chains of usual simplicial homology. This is a normed vector space for the following norm. A basis of $C_q(\Sigma)$ is obtained by choosing an oriented simplex for every $q$-simplex. Consider the $\ell^1$-norm associated to this basis, i.e., the sum of the absolute values of the coefficients in this basis; this norm does not depend on the choice of orientation since orientations only affect signs. The boundary maps $\partial \colon C_{q+1}(\Sigma)\to C_{q}(\Sigma)$ are augmented by the sum of coefficients $\epsilon \colon C_{0}(\Sigma)\to \RR$.

\medskip
The simplicial complex $\Sigma$ is \textbf{uniformly acyclic} if its (augmented) chain complex admits a contracting (chain-) homotopy $h_\bu$ such that each $h_q\colon C_q(\Sigma)\to  C_{q+1}(\Sigma)$ is bounded (in the sense of linear maps between normed vector spaces). This notion was used in various forms since~\cite{Matsumoto-Morita} and was recently systematically developed in~\cite{Kastenholz-Sroka_arx} in the semi-simplicial setting. Examples include all simplicial cones, in particular any hypersimplex, where the bound on $h_q$ can be chosen to be~$1$. More generally, two simplicial maps are \textbf{uniformly homotopic} if there exists a homotopy between the corresponding chain maps that is bounded in every degree. Accordingly, a \textbf{uniform homotopy equivalence} between $\Sigma$ and $\Sigma'$ is the data of $f\colon \Sigma\to\Sigma'$ and $f'\colon \Sigma'\to\Sigma$ such that $f\circ f'$ and $f'\circ f$ are uniformly homotopic to the identity.

For bounded cohomology (simplicial and beyond), we refer to the founding paper of Gromov~\cite{Gromov} and to~\cite{Frigerio_book,Ivanov_notes_arx_v3}. Uniform acyclicity implies by duality the vanishing of simplicial bounded cohomology (Theorems~2.3 and~2.8 in~\cite{Matsumoto-Morita}), termed \textbf{bounded acyclicity}, though the two are not equivalent.

\subsection{Support control}%%%%%%%%%%%%%%%%%%%%%%%%%%%%%%%%%%%%%%%%%%%%%%%%%%%%%%%%%%%%%%%%%%
The first tool that we introduce is a ``support-to-norm'' principle leveraging the fact that the norm on homology cycles is not just any norm, but a $\ell^1$-norm. This allows to control operator norms on a basis, which is an elementary form of projectivity; in fact, a theorem of K\"othe~\cite{Koethe66} shows that projective Banach spaces are precisely $\ell^1$-spaces.

\begin{thm}\label{thm:univ}
Let $\Sigma$ be a simplicial complex. Suppose that there is a function $\fhi\colon \NN\to\NN$ such that every subcomplex with $n$ vertices is contained in some acyclic subcomplex of $\Sigma$ with at most $\fhi(n)$ vertices.

Then $\Sigma$ is uniformly acyclic.
\end{thm}

When the hypothesis of \Cref{thm:univ} is satisfied, we say that $\Sigma$ has \textbf{support control}. Compare \Cref{defi:support} and \Cref{rem:support} below for further extensions of this definition.

\smallskip
In preparation for the proof, we introduce an auxiliary notion.

\begin{lem}\label{lem:univ}
Given $q,r\in\NN$, there is a constant $U_q(r)$ with the following property. For any simplicial complex $\Phi$ on at most $r$ vertices and for any $(q+1)$-chain $\beta\in C_{q+1}(\Phi)$, there is $\beta'\in C_{q+1}(\Phi)$ with $\partial \beta = \partial \beta'$ and $\|\beta'\| \leq U_q(r) \|\partial\beta\|$.

Moreover, there is a smallest such constant.
\end{lem}

\begin{defi}\label{defi:univ}
This smallest constant will be called the \textbf{universal constant} $U_q(r)$.

(We could choose a coarser constant depending on $r$ only since finitely many $q$ are relevant for every given $r$, but our notation allows the proof of \Cref{thm:univ} to extend to a more general setting, recorded in \Cref{thm:univ-bis} below.)
\end{defi}

\begin{proof}[Proof of \Cref{lem:univ}]
Given any finite simplicial complex $\Phi$, the homology boundary map $\partial\colon C_{q+1}(\Phi) \to C_{q}(\Phi)$ is a linear map between finite-dimensional vector spaces and therefore it is an open map (for any norm, in particular the given $\ell^1$-norm). This means that there is a constant $C$ (depending on $\Phi$ and $q$) such that any boundary $\partial \beta$, where $\beta\in C_{q+1}(\Phi)$, can be written $\partial \beta = \partial \beta'$ for a chain $\beta'$ with $\|\beta'\| \leq C \|\partial\beta\|$.

We define $U_q(r)$ as the infimum of those $C$ that have this property simultaneously for all simplicial complexes $\Phi$ on at most $r$ vertices. This is well-defined since there are only finitely many isomorphism types of such complexes.

We note that the finite dimensionality of $C_{q+1}(\Phi)$ implies also that $U_q(r)$ itself works as a constant $C$ above, despite the use of the infimum.
\end{proof}

\begin{proof}[Proof of \Cref{thm:univ}]
For brevity, we shall say that a chain $\omega\in C_q(\Sigma)$ has \textbf{support at most $m$} if $\omega$ is a linear combination of at most $m$ oriented $q$-simplices. We construct by induction on $q\geq -1$ a sequence $h_q$ of linear maps
\[
\xymatrix{0 & \RR \ar[l] \ar@<0.5ex>[r]^{h_{-1}} & C_0(\Sigma)  \ar@<0.5ex>[l]^\epsilon \ar@<0.5ex>[r]^{h_0} & C_1(\Sigma)  \ar@<0.5ex>[l]^\partial \ar@<0.5ex>[r]^{h_1} & C_{2}(\Sigma)  \ar@<0.5ex>[l]^\partial \ar@<0.5ex>[r]^{h_2} &  C_3(\Sigma)  \ar@<0.5ex>[l]^\partial \ar@<0.5ex>[r]^{h_3} & \cdots \ar@<0.5ex>[l]^\partial}
\]
and a sequence of functions $\psi_q\colon\NN\to\NN$. The inductive claims are:

\begin{enumerate}[(i)]
\item Boundedness: the linear map $h_q$ is bounded.
\item Homotopy: the identity map can be written $h_{q-1}\partial + \partial h_q$ for $q>0$, $h_{-1}\epsilon + \partial h_0$ for $q=0$ and $\epsilon h_{-1}$ for $q=-1$.
\item Support: if $q\geq0$ and $\omega\in C_q(\Sigma)$ has support at most $m\in\NN$, then $h_q(\omega)$ has support at most $\psi_q(m)$.
\end{enumerate}

To start the induction at $q=-1$, we select a vertex $v_0$ and define $h_{-1}(t) = t\{v_0\}$ for $t\in\RR$. We set $\psi_{-1}\equiv 1$. The first two inductive conditions are satisfied. The last one was not formulated for $q=-1$, but in view of the inductive step we note that its \emph{conclusion} still holds since $h_{-1}(t)$ has support at most~$1$.

We now address the inductive step for any $q\geq 0$, abusively writing $\partial$ for $\epsilon$ in the special case $q=0$.

We consider the basis of $C_q(\Sigma)$ given by some choice of an oriented simplex $\dot\sigma$ for every $q$-simplex $\sigma$. We shall first define $h_q$ on each $\dot\sigma$ (but assuming that $h_{q-1}$ is already given on its entire domain of definition).

The second inductive assumption implies that $\alpha=\dot\sigma - h_{q-1}\partial \dot\sigma$ is a cycle because
\[
\partial (h_{q-1}\partial \dot\sigma) = (\partial h_{q-1})\partial \dot\sigma = (\mathrm{Id} - h_{q-2} \partial) \partial \dot\sigma =\partial \dot\sigma
\]
for $q>0$, whereas for $q=0$ we have $\partial h_{-1}\partial \dot\sigma=\partial \dot\sigma$ from $\partial=\epsilon$.

Since $\partial \dot\sigma$ has support at most $q+1$, the third inductive assumption shows that $\alpha$ has support at most $1+\psi_{q-1}(q+1)$. Thus at most $n$ vertices are involved, where
\[
n=(q+1) \big(1+\psi_{q-1}(q+1)\big).
\]
The assumption on $\Sigma$ implies that $\alpha=\partial \beta$ for some $\beta\in C_{q+1}(\Sigma)$ such that $\beta$ (and hence also $\alpha$) is supported on a subcomplex $\Phi\se \Sigma$ with at most $\fhi(n)$ vertices. By \Cref{lem:univ}, upon possibly replacing $\beta$ by another chain in $C_{q+1}(\Phi)$, we can assume that $\beta$ has norm at most $U_q(\fhi(n)) \|\alpha\|$.

We now define $h_q(\dot\sigma)=\beta$ and extend it by linearity to define $h_q$ on all of $C_q(\Sigma)$; this is possible since the various $\dot\sigma$ form a basis. As for $\psi_q$, we define it for $m\in\NN$ by
\[
\psi_q(m) = %\binom{m}{q+1} \, \fhi\big(q+1+\psi_{q-1}(q+1)\big).
m \cdot \binom{\fhi(n)}{q+2}\, , \kern3mm \text{recalling} \kern3mm n=(q+1) \big(1+\psi_{q-1}(q+1)\big).
\]
Let us proceed to verify the inductive claims for $q$.

For the boundedness condition~(i), let $\omega\in C_q(\Sigma)$. Since $\omega$ is a finite sum of the form $\sum_{\sigma\in\Sigma_{q}} \omega(\sigma) \dot\sigma$, we have
\[
%\| h_q(\omega) \| \leq \sum_{\sigma\in\Sigma_{q}}| \omega(\sigma)| \cdot \|h_q( \dot\sigma)\| \leq \sup_{\sigma\in\Sigma_{q}}  \|h_q( \dot\sigma)\|  \sum_{\sigma\in\Sigma_{q}}| \omega(\sigma)|.
\| h_q(\omega) \| \leq \sum_{\sigma\in\Sigma_{q}}| \omega(\sigma)| \cdot \|h_q( \dot\sigma)\| \leq \Big (\sup_{\sigma\in\Sigma_{q}}  \|h_q( \dot\sigma)\|\Big)  \Big(\sum_{\sigma\in\Sigma_{q}}| \omega(\sigma)|\Big).
\]
Since $\sum_{\sigma\in\Sigma_{q}}| \omega(\sigma)|$ is the $\ell^1$-norm of $\omega$ defined on $C_q(\Sigma)$, it suffices to show that the supremum in the above expression is finite. This follows because the bound obtained above for $\beta$, namely
\[
 \|h_q(\dot\sigma) \|\leq U_q(\fhi(n))\, \|\alpha\| \leq U_q(\fhi(n)) \,\left( \|\dot\sigma\| + \| h_{q-1}\| \cdot \|\partial \|\cdot \|\dot\sigma \| \right),
\]
is independent of $\dot\sigma$, in view of the definition of $n$ and recalling that $\|\dot\sigma\|=1$.

The homotopy condition~(ii) is linear and therefore holds by construction because, on our basis,
\[
\dot\sigma = h_{q-1}\partial \dot\sigma + \alpha = h_{q-1}\partial \dot\sigma + \partial\beta = ( h_{q-1}\partial + \partial h_q) \dot\sigma.
\]
Finally, for the support condition~(iii), let $m\in\NN$ and consider $\omega\in C_q(\Sigma)$ with support at most $m$. Thus $\omega$ is a linear combination of at most $m$ oriented simplices $\dot\sigma$. For each $\dot\sigma$, our construction of $h_q(\dot\sigma)$ is a $(q+1)$-chain on a complex with at most $\fhi(n)$ vertices, and thus it has support at most $\binom{\fhi(n)}{q+2}$. It follows as claimed that $h_q(\omega)$ has support at most $\psi_q(m)$.
\end{proof}

We wrote the above proof in such a way that it shows a formally stronger statement, recorded in \Cref{thm:univ-bis} below. Only the simpler statement of \Cref{thm:univ} will be used in this article; the reader can ignore the rest of this subsection and any mention of semi-simplicial sets.

First we note that the vertex-count $\fhi$, which will be the relevant quantity in our applications to buildings, only served to bound the support, in terms of $(q+1)$-simplices, of a $(q+1)$-chain $\beta$ bounding a $q$-cycle $\alpha$ (whence the binomial coefficients). Therefore we can formalise this in the more general semi-simplicial setting where simplices are not determined by vertices:

\begin{defi}\label{defi:support}
A semi-simplicial set  $(\Sigma_q)_{q\geq 0}$ has \textbf{support control} (below some $\qtop \leq +\infty$) if for every $q<\qtop$ there is a function $\fhi_q\colon \NN\to\NN$ such that every (reduced) $q$-cycle with support at most $m$ is the boundary of a $(q+1)$-chain with support at most $\fhi_q(m)$. (The restriction $\qtop$ on the range can be useful in non-acyclic settings such as the spherical Solomon--Tits theorem.) We call $\fhi_\bu$ the \textbf{control function}.

Accordingly, we replace the universal constants $U_q(r)$ by a semi-simplicial analogue: define the \textbf{semi-simplicial universal constant} $U^{\mathrm{ss}}_q(p)$ by considering all semi-simplicial sets $\Phi_\bu$ such that $\Phi_{q+1}$ has at most $p$ elements and then considering the same constant $C$ as in the proof of \Cref{lem:univ} but for the linear map of finite-dimensional spaces $\partial\colon C_{q+1}(\Phi_\bu) \to C_{q}(\Phi_\bu)$.

In the special case of simplicial complexes, the relation to the earlier constants is thus $U_q(r) \leq U^{\mathrm{ss}}_q (\binom{r}{q+2})$ and $\fhi_q(m) \leq \binom{\fhi(m(q+1))}{q+2}$.
\end{defi}

\begin{exam}
One can check that support control below $\qtop =1$ is equivalent to: connected with finite diameter.
\end{exam}

\begin{rem}\label{rem:support}
The definition of support control is given in terms of \emph{real} cochains. When we prove it for flatmate complexes and algebraic groups, we will however establish the stronger statement that the finite complex (as in \Cref{thm:univ}) is \emph{contractible}, which implies support control for any coefficients. Accordingly, we can speak of \textbf{integral support control} for $\ZZ$ coefficients, etc.
\end{rem}

\begin{thm}[Technical variant]\label{thm:univ-bis}
Let $(\Sigma_q)_{q\geq 0}$ be a semi-simplicial set and let $1 \leq \qtop \leq +\infty$. Suppose that $\Sigma_\bu$ has support control below $\qtop$.

Then $\Sigma_\bu$ admits a bounded contracting homotopy $(h_q)_{q\geq -1}$ up to $q<\qtop$.

In particular, the real simplicial bounded cohomology $\hb^q(\Sigma_\bu)$ vanishes in degrees $0<q<\qtop$ and $\hb^{\qtop}(\Sigma_\bu)$ injects into $\hh^{\qtop}(\Sigma_\bu)$ when $\qtop<+\infty$.
\end{thm}

\begin{proof}
The inductive proof given for \Cref{thm:univ} holds almost unchanged with the adaptations introduced above. Thus, since $\alpha$ has support at most $1+\psi_{q-1}(q+1)$, it follows that $\beta$ has support at most $\fhi_q(1+\psi_{q-1}(q+1))$ and therefore the inductive definition of $\psi_q$ becomes
\[
\psi_q(m) = m \cdot \fhi_q(1+\psi_{q-1}(q+1)).
\]
The rest of the proof is unchanged.

If $\qtop\neq+\infty$, we stop the inductive argument at $q=\qtop-1$. The statements for bounded cohomology follow by duality; see Theorems~2.3 and~2.8 in~\cite{Matsumoto-Morita}.
\end{proof}

\begin{rem}
This proof shows that in hindsight we can take the control function to be linear for every given $q$ since $\psi_\bu$ is in particular also a control function.
\end{rem}

We believe that there are many circumstances where support control is a helpful method to establish (and strengthen) uniform acyclicity. For instance, it is well-suited to combinatorial arguments such as glueing:

\begin{lem}
Let $(\Sigma_q)_{q\geq 0}$ be a semi-simplicial set and let $1 \leq \qtop \leq +\infty$. Suppose that $\Sigma_\bu$ is the union of two semi-simplicial subsets $\Sigma^+_\bu$, $\Sigma^-_\bu$.

If $\Sigma^+_\bu$ and $\Sigma^-_\bu$ have support control below $\qtop$ and $\Sigma^+_\bu \cap \Sigma^-_\bu$ has support control below $\qtop-1$, then $\Sigma_\bu$ has support control below $\qtop$.

Moreover, the control function for  $\Sigma_\bu$ can be taken to depend only on the control functions for $\Sigma^\pm_\bu$ and $\Sigma^+_\bu \cap \Sigma^-_\bu$.
\end{lem}

\noindent
If the support is replaced by the norm, then the analogous statement is given in~\cite[7.13]{Kastenholz-Sroka_arx}.

\begin{proof}
A $q$-cycle $\alpha$ on $\Sigma _\bu$ can be written $\alpha^+ - \alpha^-$ for $\alpha^\pm\in C_q(\Sigma^\pm_\bu)$ without introducing new simplices in the supports. Then $\partial \alpha^+ = \partial \alpha^-$ is a $(q-1)$-cycle on  $\Sigma^+_\bu \cap \Sigma^-_\bu$. If $\partial \alpha^\pm = \partial \omega$ for $\omega\in C_q(\Sigma^+_\bu \cap \Sigma^-_\bu)$, then $\alpha^\pm - \omega$ is a $q$-cycle on $\Sigma^\pm_\bu$ and hence can be written $\partial \beta^\pm$ for $\beta^\pm\in C_{q+1}(\Sigma^\pm_\bu)$. By assumption, the supports of $\omega$ and $\beta^\pm$ can be bounded in terms of the support of $\alpha$. Finally, $\alpha= \partial (\beta^+ - \beta^-)$.
\end{proof}

\subsection{Uniform nerve principles}%%%%%%%%%%%%%%%%%%%%%%%%%%%%%%%%%%%%%%%%%%%%%%%%%%%%%%%%%%%%%%%%%%
\label{sec:nerve}%
Leray established the classical correspondence between the homotopy type of a space and that of the nerve of a good cover; a simplicial version is due to Borsuk. Our next tool is a uniform version of the simplicial nerve lemma. We write $\sub(\Sigma)$ for the \textbf{poset of subcomplexes} of a simplicial complex $\Sigma$. In order to facilitate the control of the constants, we make a strong assumption on the subcomplexes, which will be granted in our applications. We begin with a statement for finite covers (of generally infinite complexes).

\begin{thm}\label{thm:nerve}
Let $\Sigma$ be a simplicial complex and let $\sF \se \sub(\Sigma)$ be a finite cover of $\Sigma$ by subcomplexes.

Suppose that every element of $\sF$ is a hypersimplex.

Then $\Sigma$ is uniformly homotopy equivalent to the nerve complex $\nerve \sF$. Moreover, all bounds can be chosen independently of $\Sigma$ and $\sF$, i.e., they depend only on the homology degree.
\end{thm}

Note that the assumption implies that $\cap \alpha$ is also a hypersimplex in $\Sigma$ for every $\alpha\in \nerve \sF$.

The overall structure of the argument is similar to a strategy used for ordinary nerve principles such as in~\cite[Lem.~1.1]{Bjorner81}. One ingredient is the \emph{carrier lemma}, for which a uniform version was established in~\cite{FFF-monod-nariman_arx1} (in the greater generality of semi-simplicial sets). Recall that given simplicial complexes $\Omega, \Omega'$ a \textbf{carrier} is an isotone map $C\colon \Omega\to \sub(\Omega')$; it is \textbf{uniformly acyclic} if for each $q$, every complex $C(\sigma)$ is uniformly acyclic, uniformly over $\sigma\in\Omega_{q}$. A simplicial map $\fhi\colon \Omega \to \Omega'$ is \textbf{carried} by $C$ if $\forall \sigma : \fhi(\sigma )\in C(\sigma)$.

\begin{lem}[\cite{FFF-monod-nariman_arx1}]\label{lem:carrier}
If $\fhi, \psi\colon \Omega \to \Omega'$ are carried by the same uniformly acyclic carrier, then they are uniformly homotopic with constants depending only on the carrier and the homology degree.

In particular, if $\fhi, \psi\colon S\to S'$ are two isotone (or two antitone) poset maps with $\forall s: \fhi(s) \leq \psi(s)$, then the corresponding simplicial maps $\order S \to \order S'$ are uniformly homotopic with constants depending only on the homology degree.
\end{lem}

\begin{proof}[Reference for the proof]
The first statement is (the simplicial case of) Lemma~4.11 in~\cite{FFF-monod-nariman_arx1}. The second one is a version of Theorem~4.12 therein and follows from the first by exhibiting a cone as carrier.
\end{proof}

\begin{proof}[Proof of \Cref{thm:nerve}]
Considering  $\Sigma$ and $\nerve \sF$ as posets, we define an antitone map
\[
f\colon \Sigma \lra \nerve \sF,\kern3mm f(\sigma) = \{ F\in \sF : \sigma \in F\}.
\]
Next, for every $\alpha\in\nerve \sF$, we choose some vertex $x_\alpha$ of $\cap \alpha$. Define
\[
g\colon \nerve \sF \lra\Sigma,\kern3mm g(\beta) = \big \{ x_\alpha : \alpha\in\nerve\sF \text{ with } \beta \se \alpha\big\}.
\]
Note that $g(\beta)$ is indeed a simplex of $\Sigma$ since all those $x_\alpha$ are in $\cap\beta$, which is a hypersimplex. The map $g$ is an antitone map of posets.

We can now consider $f$ and $g$ as simplicial maps between the corresponding order complexes $\order\Sigma$ and $\order\nerve\sF$, which are none other than the barycentric subdivisions of $\Sigma$ and $\nerve\sF$. We claim that these simplicial maps are uniform homotopy inverses to each other with all bounds depending only on the homology degree. This claim will complete the proof of the theorem because any simplicial complex is uniformly homotopy equivalent to its barycentric subdivision with bounds depending only on the homology degree. Indeed, the classical homotopy equivalences are given by explicit sums depending only on the degree; see e.g., \cite[\S17]{Munkres}. (This fact has also been established in~\cite[\S7.7]{Kastenholz-Sroka_arx} for the generality of semi-simplicial sets.)

Turning to the claim, we consider the composition $g f$. This is an isotone map on the poset $\Sigma$ (the vertex set of $\order\Sigma$). We have
\[
g f (\sigma) = \big\{ x_\alpha : \alpha\in\nerve\sF \text{ such that } \forall F\in \sF, \sigma\in F \Rightarrow F \in \alpha \big\}
\]
and $g f (\sigma)$ is a simplex of the subcomplex $\cap f(\sigma)$ of $\Sigma$. Note that the map which to each $\sigma$ associates $\cap f(\sigma)$ is a carrier from $\Sigma$ to itself. We define a further carrier map
\[
C\colon \order\Sigma\lra \sub( \order\Sigma), \kern3mm C(\{\sigma_0 \subsetneqq \cdots \subsetneqq \sigma_p\}) = \order \big(\cap f(\sigma_p)\big).
\]
This carrier $C$ carries the simplicial map $g f$; indeed:
\[
g f (\{\sigma_0 \subsetneqq \cdots \subsetneqq \sigma_p\}) =  \{g f(\sigma_0) \se \cdots \se g f(\sigma_p)\} \in \order \big(\cap f(\sigma_p)\big)
\]
since each  $g f(\sigma_j)$ is in $\cap f(\sigma_j)$ which is a subset of $\cap f(\sigma_p)$. On the other hand, $C$ also carries the identity because $\sigma_p\in\cap f(\sigma_p)$ by definition of $f$. In order to conclude from the uniform carrier lemma that $g f$ and the identity are uniformly homotopy equivalent with the desired uniformity of constants, it remains only to justify that the carrier $C$ is uniformly acyclic with constants depending only on the degree. But this last point follows from the fact that the carrier is the barycentric subdivision of $\cap f(\sigma_p)$, which is a hypersimplex.

We now consider the other composition, $f g$, which is simpler. We have
\[
f g(\beta) = \big\{ F\in \sF: \forall \alpha\in\nerve\sF, \beta\se\alpha \Rightarrow x_\alpha\in F \big\}.
\]
Unravelling all definitions, we see that $\beta \se f g (\beta)$. In other words, $f g$ dominates the identity (as poset maps); therefore, the uniform carrier lemma applies.
\end{proof}

The covers that will appear in the proof of our main result are not finite, in fact not even locally finite (they have locally the power of continuum), but turn out to have uniformly acyclic nerves. Therefore, we shall need the following variant of \Cref{thm:nerve}. We caution the reader that a difficulty resides in the fact that successive nested finite subcovers will a priori give distinct homotopy equivalence maps; we will argue that they must be uniformly homotopic to each other.

\begin{thm}\label{thm:nerve:infinite}
Let $\Sigma$ be a simplicial complex and let $\sF \se \sub(\Sigma)$ be a cover of $\Sigma$ by subcomplexes.

Suppose that every element of $\sF$ is a hypersimplex.

Then $\Sigma$ is uniformly acyclic if and only if the nerve $\nerve \sF$ is so.
\end{thm}

\begin{proof}
Suppose that $\nerve \sF$ is uniformly acyclic. Fix $q\in \NN$ and consider any cycle $\omega\in C_q(\Sigma)$. Then there is a finite subset $\sF'\se\sF$ which covers all simplices involved in $\omega$. We consider the subcomplex $\Sigma'$ of $\Sigma$ given by the union of $\sF'$ and view $\omega$ as a cycle on  $\Sigma'$. The homotopy equivalence of \Cref{thm:nerve}, applied to $\Sigma'$,  means that there are bounded linear chain maps  $C_q(\Sigma') \to C_q(\nerve\sF')$ and in the opposite direction which induce mutually inverse isomorphisms in homology. Consider the cycle $\eta\in C_q(\nerve\sF')$ corresponding to $\omega$ as a cycle on $\nerve\sF$. The uniform acyclicity assumption implies that we can write $\eta = \partial \vartheta$ for $\vartheta\in C_{q+1}(\nerve\sF)$ with $\|\vartheta\|\leq c \|\eta\|$, where $c$ depends only on $q$ and on $\nerve\sF$.

Again $\vartheta$ is supported on $\nerve\sF'' \se \nerve\sF$ for some finite $\sF'' \se\sF$ and we can assume that $\sF''$ contains $\sF'$. We apply again \Cref{thm:nerve}, this time to $\Sigma''=\cup\sF''$, noting that $\omega$ is also a cycle for this complex. The resulting chain map $C_q(\Sigma'') \to C_q(\nerve\sF'')$ sends $\omega$ to some cycle $\widehat\eta\in C_q(\nerve\sF'')$. We claim that $\widehat\eta -\eta = \partial \fhi$ for some $\fhi\in C_{q+1}(\nerve\sF'')$ with $\|\fhi\|\leq c' \|\omega\|$, where $c'$ depends only on $q$, on $\Sigma$ and on $\nerve\sF$. This claim will finish the proof that $\Sigma$ is uniformly acyclic, since $\widehat\eta = \partial (\vartheta+\fhi)$ will then imply that $\omega$ is a boundary (in $\Sigma''$) and since all constants (including those from \Cref{thm:nerve}) depend on $q$, $\Sigma$ and $\sF$ only, not on $\omega$.

To justify the claim, we need to compare the two maps
\[
C_q(\Sigma') \lra C_q(\nerve\sF') \ \se\ C_q(\nerve\sF'') \kern2mm \text{and}\kern2mm C_q(\Sigma')  \ \se\  C_q(\Sigma'') \lra C_q(\nerve\sF'') 
\]
which produce the cycles $\eta$, respectively, $\widehat\eta$, from $\omega$. It suffices to show that these maps are homotopic with uniform constants.

Consider the underlying two poset maps
\[
f'\colon \Sigma' \lra \nerve\sF' \ \se\ \nerve\sF'' \kern2mm \text{and}\kern2mm f''\colon \Sigma' \ \se\  \Sigma'' \lra \nerve\sF''
\]
which are constructed in the proof of \Cref{thm:nerve}. Strictly speaking, the first one is the corestriction of the map $f$ constructed for $\Sigma'$, the second the restriction of the map $f$ for $\Sigma''$. Appealing again to the uniform carrier lemma in the form of \Cref{lem:carrier}, it suffices to show the following: for every $\sigma\in\Sigma'$, $f'(\sigma) \se f''(\sigma)$. This, however, is apparent in the definition given for $f$ in the proof of \Cref{thm:nerve}.

\smallskip
The converse, which we will not need, is proved in exactly the same way. Namely, given finite subcovers $\sF' \se \sF'' \se\sF$ and the corresponding subcomplexes $\Sigma' \se \Sigma''$, it suffices to compare the two poset maps 
\[
g'\colon \nerve\sF' \lra \Sigma'  \ \se\  \Sigma'' \kern2mm \text{and}\kern2mm g''\colon  \nerve\sF' \ \se\ \nerve\sF'' \lra  \Sigma'' 
\]
arising from the proof of \Cref{thm:nerve}. If the choice $\alpha\mapsto x_\alpha$ has been fixed once and for all for every $\alpha\in \nerve\sF$, then indeed $g'(\beta) \se g''(\beta)$ holds for all $\beta\in\nerve\sF'$ and we conclude as above.
\end{proof}

%\subsection{}%%%%%%%%%%%%%%%%%%%%%%%%%%%%%%%%%%%%%%%%%%%%%%%%%%%%%%%%%%%%%%%%%%

\section{The flatmate complex of Euclidean buildings}%%%%%%%%%%%%%%%%%%%%%%%%%%%%%%%%%%%%%%%%%%%%%%%%%%%%%%%%%%%%%%%%%%
%%%%%%%%%%%%%%%%%%%%%%%%%%%%%%%%%%%%%%%%%%%%%%%%%%%%%%%%%%%%%%%%%%%%%%%%%%%%%%%%%%%%
\subsection{Euclidean buildings}%%%%%%%%%%%%%%%%%%%%%%%%%%%%%%%%%%%%%%%%%%%%%%%%%%%%%%%%%%%%%%%%%%
\label{sec:affine}%
We shall adopt the viewpoint that a building is a complex endowed with a system of apartments and refer to~\cite{Abramenko-Brown}, \cite{Ronan_buildings2} and~\cite{Weiss_affine} for background. For simplicity, we only consider \emph{irreducible} buildings, which are therefore simplicial (rather than polysimplicial) complexes. All arguments below adapt to the non-irreducible case, but this setting is not needed for our vanishing results because, as we shall recall in \Cref{sec:vanishing}, the vanishing passes to finite products of groups (and from finite index subgroups).

\smallskip
In the case of discrete irreducible \emph{Euclidean} buildings, each apartment is a full subcomplex isomorphic to a triangulation of a Euclidean space, making the synonym `flat' especially congruous. We recall that the apartments are \textbf{combinatorially convex}, which means by definition that every minimal gallery connecting two chambers of an apartment remains in that apartment; see Prop.~4.40 in~\cite{Abramenko-Brown}.

\begin{prop}\label{prop:add-flats}
Given a discrete irreducible Euclidean building, there exists an integer $k$ with the following property.

For every $n\in \NN$ and every family of $n$ apartments $F_1, \ldots, F_n$, there exists a family of $k n$  apartments $E_1, \ldots, E_{k n}$ such that the union
\[
%\bigcup_{i=1}^n F_i \ \cup \ \bigcup_{j=1}^{k n} E_j
F_1 \cup \cdots \cup F_n \cup E_1 \cup \cdots \cup E_{k n}
\]
is contractible (as a simplicial complex).
\end{prop}

\begin{proof}
We define $k$ to be the number of chambers of the spherical Coxeter complex associated to the building. Specifically, we realise it as the number of chambers at infinity of any apartment.

Fix some chamber $c$ of the Euclidean building. Given any apartment $F$ and any chamber $\xi$ of the apartment at infinity $\partial F$, choose some apartment $E_\xi$ containing $c$ with $\xi\in \partial E_\xi$; for the existence of such $E_\xi$, see Prop.~7.6 in~\cite{Weiss_affine}. We claim that the union of those $k$ apartments $E_\xi$ contains $F$.

Indeed, select a special vertex $y$ of $c$ and consider the sector $S_\xi$ based at $y$ and representing $\xi$. The combinatorial convexity of apartments implies that $S_\xi$ is contained in $E_\xi$. However, it is known that the union of the sectors $S_\xi$ contains $F$ as $\xi$ ranges over the chambers of $\partial F$; this holds in the greater generality of possibly non-discrete Euclidean buildings, see, e.g., the proof of~\cite[Lem.~6.3]{Hitzelberger2011} or of~\cite[Prop.~7.3]{Bennett-Schwer}. This justifies the claim.

\smallskip
Returning to the statement of the proposition, we define the family $E_j$ as the collection of all $E_\xi$ chosen as above for each $F=F_1, \ldots, F_n$. By the claim, the union in the statement of the proposition reduces to the union of all $E_j$. That union is \textbf{combinatorially starlike} with respect to $c$, which by definition means the following: any minimal gallery from $c$ to any chamber in the union remains in this union. Indeed, this holds by combinatorial convexity of the apartments since every $E_j$ contains $c$. It only remains to note that, in a Euclidean building, combinatorially starlike chamber subcomplexes are contractible. This is Exercise~4.130 in~\cite{Abramenko-Brown} and it holds by exactly the same shellability argument as used for the Solomon--Tits theorem to obtain the contractibility of the Euclidean building itself; see~\cite[Thm.~4.127]{Abramenko-Brown} or~\cite[App.~II]{Garland}.
\end{proof}

\subsection{The flatmate complex}%%%%%%%%%%%%%%%%%%%%%%%%%%%%%%%%%%%%%%%%%%%%%%%%%%%%%%%%%%%%%%%%%%
\label{sec:flatmate}%
We begin with one more general construction of simplicial complexes. Let $X$ be a set and $\sF$ a cover of $X$. Given $F\in \sF$, the hypersimplex $\simplex F$ is a full subcomplex of $\simplex X$. We can therefore define a subcomplex $\simplex_\sF X$ of $\simplex X$ by
\[
\simplex_\sF X = \bigcup_{F\in\sF} \simplex F.
\]
We now apply the nerve principle of \Cref{thm:nerve:infinite} to these complexes.

\begin{cor}\label{cor:abstract-flatmate}
Let $X\neq \varnothing$ be a set and let $\sF$ be a cover of $X$ by non-empty subsets.

If $\nerve \sF$ is uniformly acyclic, then so is $\simplex_\sF X$.
\end{cor}

\begin{proof}
By definition, the complex $\simplex_\sF X$ is covered by the family of subcomplexes $\simplex F$ as $F$ ranges over $\sF$. For any collection $F_1, \ldots, F_n$ of elements of $\sF$, we have
\[
\simplex\big(F_1\cap \ldots \cap F_n\big) = \simplex(F_1)\cap \ldots \cap \simplex(F_n).
\]
Thus the nerve of this cover of $\simplex_\sF X$ is canonically isomorphic to the nerve $\nerve \sF$ of the cover of $X$. We are therefore indeed in the situation of \Cref{thm:nerve:infinite}.
\end{proof}

We now specialise to buildings and flats.

Consider a building $\sB$ with apartment system $\sA$. The apartment system provides in particular a cover $\sF$ of the set $X$ of vertices of $\sB$; formally,
\[
\sF = \big\{ F= \vertex(A) : A \in \sA \big\} \kern3mm\text{covers} \kern3mm X= \vertex(\sB).
\]
The following is the simplicial form of the algebraic definition in terms of chain groups that we proposed with Bucher in~\cite[Scholium]{Bucher-Monod_tree}.

\begin{defi}
The \textbf{flatmate complex} of the building $\sB$ is the simplicial complex $\simplex_\sF X$ as defined above. We also refer to it, in Tits's tongue, as the \textbf{cokotcomplex} of $\sB$.
\end{defi}

We can now answer the problem suggested in~\cite{Bucher-Monod_tree}, as announced in \Cref{ithm:flatmate}. We restate it here since the notation has now been introduced:

\begin{thm}\label{thm:flatmate}
The flatmate complex of any discrete irreducible Euclidean building is uniformly acyclic.
\end{thm}

The remaining ingredient for the proof of this theorem is as follows.

\begin{thm}\label{thm:build-nerve}
Let $\sB$ be a discrete irreducible Euclidean building with apartment system $\sA$. Then the nerve $\nerve\sA$ is uniformly acyclic.
\end{thm}

\begin{proof}
We shall argue that the simplicial complex $\Sigma=\nerve\sA$ satisfies the assumption of \Cref{thm:univ} for the function $\fhi(n) = (k +1)n$, where $k$ is the apartment size of the associated spherical building as in \Cref{sec:affine}. To that end, note that for any non-empty subset $\sA_0\se \sA$, the nerve $\nerve\sA_0$ is a subcomplex of $\Sigma$. In view of \Cref{prop:add-flats}, we only need to justify the following claim:

Given any non-empty finite subset $\sA_0$ of $\sA$, we consider the simplicial subcomplex $\sB_0$ of $\sB$ covered by $\sA_0$. The claim is that if the simplicial complex $\sB_0$ is contractible, then so is the nerve $\nerve\sA_0$.

To justify the claim, we appeal to the nerve lemma in \emph{ordinary} simplicial homology. This is sometimes called the Borsuk nerve lemma; in precisely the setting of abstract simplicial complexes, two proofs can be found in~\cite[Lem.~1.1]{Bjorner81}. That lemma asserts that $\sB_0$ and $\nerve\sA_0$ have the same homotopy type provided that every non-empty intersection of subcomplexes taken from the family $\sA_0$ is contractible. Such an intersection is a subcomplex of the building $\sB$, namely a non-empty intersection of apartments. The combinatorial convexity of apartments implies that this intersection is contractible and therefore the claim is established.
\end{proof}

\begin{proof}[End of proof of \Cref{thm:flatmate} (i.e., \Cref{ithm:flatmate})]
The nerve $\nerve\sF$ coincides with the nerve $\nerve\sA$ of the cover of the building $\sB$ by its apartments. Thus \Cref{thm:build-nerve} states that $\nerve\sF$ is uniformly acyclic. Therefore, \Cref{cor:abstract-flatmate} implies indeed that the flatmate complex $\simplex_\sF X$ is uniformly acyclic.
\end{proof}

\section{Vanishing theorems}%%%%%%%%%%%%%%%%%%%%%%%%%%%%%%%%%%%%%%%%%%%%%%%%%%%%%%%%%%%%%%%%%%
%%%%%%%%%%%%%%%%%%%%%%%%%%%%%%%%%%%%%%%%%%%%%%%%%%%%%%%%%%%%%%%%%%%%%%%%%%%%%%%%%%%%
\label{sec:vanishing}%
We first recall some terminology. A locally compact group $G$ is \textbf{boundedly acyclic} (as a topological group) if its continuous bounded cohomology with real coefficients $\hbc^n(G)$ vanishes in every degree $n>0$.

A topological group is \textbf{amenable} if every jointly continuous affine $G$-action on any non-empty convex compact set (in any Hausdorff locally convex topological vector space) admits a fixed point. This holds notably when $G$ is compact or soluble (e.g., abelian), and is preserved by group extensions.

A subgroup $H<G$ of the topological group $G$ is \textbf{co-amenable} in $G$ if $G$ has the above fixed-point property for the subclass of those convex compact sets having an $H$-fixed point. This holds for instance if $H$ has finite index, or if $H$ is normal with $G/H$ amenable.

\medskip
We shall use the well-known general principles summarised in the proposition below to reduce the proof of \Cref{ithm:algebraic} from general algebraic groups to the simple case. 

\begin{prop}\label{prop:reduc}
A locally compact group $G$ is boundedly acyclic in each of the following cases:
\begin{enumerate}[(i)]
\item $G$ admits a co-amenable closed subgroup which is boundedly acyclic;
\item $G$ admits an amenable normal closed subgroup $N\lhd G$ with $G/N$ boundedly acyclic;
\item $G$ is the direct product of finitely many boundedly acyclic groups;
\item $G$ is the quotient of a boundedly acyclic group by an amenable normal closed subgroup.
\end{enumerate}
\end{prop}

\begin{proof}
For~(i), see~\cite[Prop.~8.6.6]{Monod}. For~(ii) and~(iv), see~\cite[Cor.~8.5.2]{Monod}. For~(iii), combine~\cite[Prop.~12.2.1]{Monod} with ~\cite[Prop.~12.2.2(ii)]{Monod}. These references make a second countability assumption which is not necessary for real coefficients (but in our case all algebraic groups are second countable anyway).
\end{proof}

\subsection{Automorphism groups of buildings}%%%%%%%%%%%%%%%%%%%%%%%%%%%%%%%%%%%%%%%%%%%%%%%%%%%%%%%%%%%%%%%%%%
Recall that a group of building automorphisms is \textbf{strongly transitive} if it acts transitively on the set of pairs consisting of a chamber and an apartment containing it.\footnote{It is ironic that the term Tits chose for his marvelous concept of building is \emph{immeuble} --- literally: \emph{that which cannot be moved} --- whereas he demonstrated how deeply buildings are entwined with their rich transformation groups.} The fact that \Cref{ithm:building} should follow from \Cref{thm:flatmate} was introduced in~\cite{Bucher-Monod_tree}. We shall nonetheless give all details of the argument since the language is somewhat different here.

\begin{proof}[Proof of \Cref{ithm:building}]
Let $G$ be a locally compact group with a strongly transitive proper action by automorphisms on a locally finite Euclidean building $\sB$ with apartment system $\sA$. In order to prove the vanishing of the bounded cohomology $\hbc^n(G)$ for all $n>0$, we can assume that $\sB$ is irreducible and hence fits the assumptions of \Cref{thm:flatmate}. Indeed, the automorphism group of a product admits the product of the automorphism groups of the factors as a finite index subgroup.

We therefore consider the flatmate complex $\simplex_\sF X$ on $X=\vertex\sB$ defined in \Cref{sec:flatmate} and note that $G$ acts on it by simplicial automorphisms.
The bounded simplicial cochains of this flatmate complex yield an augmented cochain complex
\[
%\xymatrix{0 & \RR \ar[l] & {\ell^\infty \big((\simplex_\sF X)_{0}\big)} \ar[l]_-{\epsilon}  & {\ell^\infty \big((\simplex_\sF X)_{1}\big)} \ar[l]_-\partial &% {\ell^\infty \big((\simplex_\sF X)_{2}\big)} \ar[l]_\partial &
% \cdots \ar[l]_-\partial}
\xymatrix{0 \ar[r] & \RR \ar[r] & {\ell^\infty \big((\simplex_\sF X)_{0}\big)}  \ar[r] & {\ell^\infty \big((\simplex_\sF X)_{1}\big)}  \ar[r] &% {\ell^\infty \big((\simplex_\sF X)_{2}\big)} \ar[r]&
 \cdots }
\]
In terms of vertices, a $q$-cochain $f\in \ell^\infty \big((\simplex_\sF X)_{q}\big)$ is a bounded alternating function on $(q+1)$-tuples of vertices of the building, where each tuple is restricted to lie in some apartment. (This is the viewpoint adopted in~\cite{Bucher-Monod_tree}.)

\Cref{thm:flatmate} above implies that this cochain complex is acyclic since it is the norm dual of the normed chain complex, which is uniformly acyclic by \Cref{thm:flatmate} (this is the duality of vanishing introduced by~\cite{Matsumoto-Morita}). On the other hand, this dual cochain complex is a $G$-complex of Banach $G$-modules and thus it is a resolution of $\RR$ in the sense of continuous bounded cohomology~\cite{Monod}.

The properness assumption implies that for each $q\geq 0$ the $G$-action on any set of $(q+1)$-tuples is also proper; this guarantees that each of the modules $\ell^\infty \big((\simplex_\sF X)_{q}\big)$ is relatively injective in the sense of continuous bounded cohomology; see~\cite[Thm.~4.5.2]{Monod}. It follows that the continuous bounded cohomology of $G$ is canonically realised by the subcomplex of $G$-invariant functions in $\ell^\infty \big((\simplex_\sF X)_{q}\big)$; see e.g., \cite[Thm.~7.2.1]{Monod}.

We continue along the lines that we proposed with Bucher in~\cite{Bucher-Monod_tree}. Choose some apartment $E\in \sA$ and denote by $H<G$ its (set-wise) stabiliser in $G$. The restriction to tuples in $E$ is a chain map
\[
%\ell^\infty \big((\simplex_\sF X)_{q}\big)^G \lra \ell^\infty \big( (E_{0})^{q+1}\big)^H
\ell^\infty \big((\simplex_\sF X)_{q}\big)^G \lra \ell^\infty_\mathrm{alt} \big( (\vertex E)^{q+1}\big)^H
\]
where $\ell^\infty_\mathrm{alt}$ denotes bounded alternating functions. The claim is that this map is bijective, thus establishing an isomorphism between the continuous bounded cohomology of $G$ and of $H$. This will complete the proof because $H$ is an amenable group and is therefore boundedly acyclic.

The injectivity of the claim follows from the transitivity of $G$ on $\sA$ and the definition of the flatmate complex $\simplex_\sF X$. Surjectivity is the only time where we are using the strong transitivity, as follows.

Pick $f$ in $\ell^\infty_\mathrm{alt}  ( (\vertex E)^{q+1})^H$ and let $x$ be any $(q+1)$-tuple of vertices in any apartment (i.e., $x$ represents any simplex of the flatmate complex). We know that $gx \se E$ for some $g\in G$ and we want to extend $f$ to this $x$ by setting $f(x) = f(g x)$. To show that this is well-defined and that the resulting function $f$ on $(\simplex_\sF X)_{q}$ is $G$-invariant, the only point to verify is that any other $g'\in G$ with $g' x \se E$ satisfies $f(g' x) = f(gx)$.

Consider the two apartments $g\inv E$ and ${g'}\inv E$; both contain $x$. By strong transitivity, there exists $q\in G$ which maps $g\inv E$ to ${g'}\inv E$ and such that $q$ fixes pointwise the intersection $g\inv E \cap {g'}\inv E$; see~\cite[Prop.~6.6]{Abramenko-Brown}. In particular, $q$ fixes every vertex in $x$. Now $h=g' q g\inv$ is an element of $H$ and $h g x = g' q x = g'x$. Since $f$ was supposed $H$-invariant on tuples in $E$, it follows $f(g x) = f(g' x)$ as desired.
\end{proof}

\begin{rem}
In view of potential generalisations, we point out that the above argument used only the bounded acyclicity, rather than the amenability, of the apartment stabiliser $H$.
\end{rem}

\subsection{Algebraic groups}%%%%%%%%%%%%%%%%%%%%%%%%%%%%%%%%%%%%%%%%%%%%%%%%%%%%%%%%%%%%%%%%%%
%\label{sec:algebraic}%
We now consider an arbitrary algebraic group $\GG$ over a non-Archimedean local field $k$.

\begin{proof}[Proof of \Cref{ithm:algebraic}]
 Our goal is to show that the locally compact group $G=\GG(k)$ is boundedly acyclic. We recall here that $G$ is endowed with the canonical Hausdorff ``strong'' topology~\cite[I\S10]{Mumford_red99}.

We use general structure theory to reduce $\GG$ to a simpler class of groups, the class of simple groups. We need however to keep track of $G=\GG(k)$ in view of the discrepancy between quotients of algebraic groups and of the corresponding $k$-points. 

\medskip
There is no loss of generality in assuming $\GG$ connected since passing to $\GG^0$ will replace $G$ by a finite index closed subgroup, which is fine by \Cref{prop:reduc}(i).

We first recall that there is a canonical \textbf{affinisation} quotient, i.e.,  a faithfully flat morphism $\psi\colon \GG\to\GG_{\mathrm{aff}}$ to an affine group $\GG_{\mathrm{aff}}=\mathrm{Spec}\sO(\GG)$ and that the kernel $\Ker\psi$ is contained in the center of $\GG$. This is a general form of a theorem of Rosenlicht~\cite{Rosenlicht56} established in~\cite[III.3.8]{Demazure-Gabriel}: combine Thm.~8.2 and Cor.~8.3 therein under our assumption $\GG=\GG^0$. It can also be read in~\cite[Thm.~1]{Brion_AG}.

In particular, $\Ker\psi$ is commutative and $\GG_{\mathrm{aff}}$, being affine, is a linear algebraic group~\cite[\S3.4]{Waterhouse}.

The quotient of $\GG_{\mathrm{aff}}$ by its radical is a connected semi-simple group $\SSS$. (In positive characteristic, it might be pseudo-semi-simple, which can be bypassed either by taking a finite field extension or by using below Solleveld's extension~\cite{Solleveld} of Bruhat--Tits buildings to the pseudo-semi-simple case.) Thus $\SSS$ is the almost-direct product of some number $r$ of (quasi-)simple connected factors $\GG_i$, $i=1, \ldots r$. We reorder the factors so that $\GG_i$ is $k$-isotropic exactly when $i\leq s$ for some $0\leq s \leq r$. 

\smallskip

Recall that $\GG_i(k)^+$ denotes the normal subgroup of $\GG_i(k)$ generated by the $k$-points of the $k$-split unipotent subgroups of $\GG_i$. This group is introduced in detail by Borel--Tits~\cite[\S6]{Borel-Tits73}; several equivalent definitions are given in~\cite[6.2]{Borel-Tits73}. In many cases (including local fields of characteristic zero), $\GG_i(k)^+ = \GG_i(k)$ holds for isotropic groups; in the general case, we shall use that the quotient $\GG_i(k) / \GG_i(k)^+$ is compact~\cite[6.14]{Borel-Tits73}.

Consider first $i\leq s$. Then Bruhat--Tits theory~\cite{Bruhat-Tits1972,Bruhat-Tits1984} (as above we refer to~\cite{Abramenko-Brown}, \cite{Ronan_buildings2} and~\cite{Weiss_affine} for background) shows that the locally compact group $\GG_i(k)$ acts strongly transitively on an irreducible locally finite Euclidean building. The resulting action of $\GG_i(k)^+$ is still strongly transitive; this follows from the decomposition given in~\cite[6.11(i)]{Borel-Tits73}. Moreover this action is proper since the center of $\GG_i(k)$ is finite. Therefore, \Cref{ithm:building} implies that $\GG_i(k)^+$ is boundedly acyclic for all $i\leq s$. If $i>s$, then by convention $\GG_i(k)^+$ is trivial. In conclusion, \Cref{prop:reduc}(iii) allows us to obtain that the product $\prod_{i=1}^r \GG_i(k)^+$ is boundedly acyclic. It follows by \Cref{prop:reduc}(iv) that $\SSS(k)^+$ is also boundedly acyclic because $\SSS(k)^+$ is the almost-direct product of the $\GG_i(k)^+$; see~\cite[6.2(iii)]{Borel-Tits73}.

\smallskip
At this point we consider the continuous group homomorphism
\[
f\colon G = \GG(k) \lra \GG_{\mathrm{aff}}(k) \lra \SSS(k) =  \GG_1(k)  \cdots  \GG_r(k).
\]
We claim that the image $f(G)$ in $\SSS(k)$ contains $\SSS(k)^+$. Indeed this image is a Zariski-dense normal subgroup; therefore it must contain each $\GG_i(k)^+$ since the latter is abstractly simple modulo its center by the main result of~\cite{Tits64}. The claim follows.

We deduce that the pre-image $G^+<G$ of $\SSS(k)^+$ is a normal cocompact subgroup of $G$. In particular it is co-amenable and therefore, by \Cref{prop:reduc}(i), it suffices to show that $G^+$ is boundedly acyclic. Since we already know that  $\SSS(k)^+$ is boundedly acyclic, this follows from \Cref{prop:reduc}(ii) if we justify that the kernel of $f|_{G^+}$ is amenable. By construction, this kernel is contained in a central extension of the group of $k$-points of the radical of $\GG_{\mathrm{aff}}$; therefore it is soluble and this completes the proof.
\end{proof}

\begin{rem}
$\Ker\psi$ and the radical of $\GG_{\mathrm{aff}}$ could have been combined into one ``radical of $\GG$'' in the above reductions, but the author is more comfortable separating the two steps since the classical structure theory is often stated in the context of \emph{linear} algebraic groups~\cite{Borel91,Humphreys,Springer98}.
\end{rem}

We still need to justify \Cref {iprop:SL2}, for which we claim no originality: the idea is taken from the introduction of~\cite{Burger-MonodERN}, but translated to the non-Archimedean context.

\begin{proof}[Proof of \Cref {iprop:SL2}]
Let $G=\SL_2(\QQ_p)$ and let $\Gamma$ be a (finite rank non-abelian) free group realised as a cocompact lattice in $G$. Then $\hb^2(\Gamma)$ is non-trivial: this was already established by Johnson in~\cite[Prop.~2.8]{Johnson} and rediscovered by Brooks~\cite[\S3]{Brooks}. The coefficient induction of bounded cohomology~\cite[\S10.1]{Monod} implies that $\hcb^2(G, L^\infty(G/\Gamma))$ is also non-trivial. Since we are in degree two, a ``double ergodicity with coefficients'' argument shows that $\hcb^2(G, L^2(G/\Gamma))$ is non-vanishing: this was established in~\cite{Burger-Monod1}, see also \cite[Cor.~11]{Burger-Monod3}.

Consider now a direct integral decomposition of $L^2(G/\Gamma)$ into irreducible continuous unitary representations of $G$, which is even a Hilbertian sum decomposition in this case~\cite[Thm.~9.2.2]{Deitmar-Echterhoff}. Appealing again to double ergodicity with coefficients, we conclude that some of these representations must have non-trivial $\hcb^2$ (see, e.g., Thm.~3.3 and Cor.~3.4 in~\cite{Monod-Shalom2}; these are stated for discrete groups but hold verbatim in the locally compact case).
\end{proof}

\begin{rem}
The first paragraph of the above argument can be replaced with a geometric construction of a cocycle with coefficients in a multiple of the regular representation of $G$, as explained in~\cite[\S2]{Monod-ShalomCRAS} and~\cite[\S4]{Monod-Shalom1}. Then the second paragraph holds using a direct integral decomposition (the Plancherel decomposition).
\end{rem}

\subsection{Arithmetic groups}%%%%%%%%%%%%%%%%%%%%%%%%%%%%%%%%%%%%%%%%%%%%%%%%%%%%%%%%%%%%%%%%%%
We retain the notation of \Cref{ithm:Sadic} and we refer to~\cite[\S I.3]{Margulis} for background on $S$-arithmetic groups, notably the following few facts:

Given $v\in S$ we denote by $K_v$ the corresponding completion and recall that $\GG(K_v)$ is non-compact precisely when $\GG$ is $K_v$-isotropic, i.e., when $\mathrm{rank}_{K_v}(\GG)>0$. Let $S_1\se S$ be the collection of those isotropic valuations in $S$; we can assume $S_1\neq \emptyset$ since otherwise the statement is void.  The assumption $S_0\se S$ implies that $\Gamma$ is a lattice in $\prod_{v\in S_1} \GG(K_v)$ and this lattice is irreducible by the strong approximation theorem. We also note at this point that \Cref{ithm:Sadic:0} is indeed a particular case of \Cref{ithm:Sadic}.

\begin{proof}[End of the proof of \Cref{ithm:Sadic}]
Applying~\cite[Cor.~1.4]{MonodVT}, we deduce that the restriction map from the full product of Archimedean as well as non-Archimedean groups
\[
\hbc^n\big(\prod_{v\in S_1} \GG(K_v)\big) \lra \hb^n(\Gamma)
\]
is an isomorphism for all $n< 2 \sum_{v\in S} \mathrm{rank}_{K_v}(\GG)$. Next we observe that the restriction map from the Lie group considered in \Cref{ithm:Sadic} is induced by the composition
\[
%\xymatrix{\Gamma \ar@{^{(}->}[r] & \prod_{v\in S_1} G(K_v) \ar@{->>}[r] & \prod_{v\in S_0} G(K_v) }
%
\Gamma \lra  \prod_{v\in S_1} \GG(K_v) \lra \prod_{v\in S_0} \GG(K_v) = L,
\]
where the right arrow is the projection. Therefore, what is needed is to prove that the inflation map corresponding to that projection is an isomorphism. \Cref{ithm:algebraic} implies that this is true in all degrees, using the Hochschild--Serre sequence; specifically, the statements of Prop.~12.2.1 and Prop.~12.2.2(ii) in~\cite{Monod}. 
\end{proof}

%===========================bibliographie===================================

\bibliographystyle{../BIB/amsalpha-nobysame}
\bibliography{../BIB/ma_bib}

\end{document}